\documentclass[12pt]{amsart}
\usepackage{graphicx}
\usepackage{tikz}
\begin{document}

\newtheorem{theorem}{Theorem}[section]
\newtheorem{lemma}[theorem]{Lemma}
\newtheorem{proposition}[theorem]{Proposition}
\newtheorem{corollary}[theorem]{Corollary}

\theoremstyle{definition}
\newtheorem{definition}[theorem]{Definition}
\newtheorem{example}[theorem]{Example}
\newtheorem{xca}[theorem]{Exercise}
\newtheorem{notation}[theorem]{Notation}

\theoremstyle{remark}
\newtheorem{remark}[theorem]{Remark}


\newcommand{\NN}{ {\mathbb N} }
\newcommand{\ZZ}{ {\mathbb Z} }
\newcommand{\RR}{ {\mathbb R} }
\newcommand{\CC}{{\mathbb C}}
\newcommand{\cD}{{\mathcal D}}
\newcommand{\cP}{ {\mathcal P} }
\newcommand{\cPP}{{\mathcal{PS}}}
\newcommand{\cF}{{\mathcal F}}
\newcommand{\HH}{{\mathcal H}}
\newcommand{\GG}{{\mathcal G}}
\newcommand{\cV}{{\mathcal V}}
\newcommand{\cW}{{\mathcal W}}
\newcommand{\cM}{{\mathcal M}}
\newcommand{\KK}{{\kappa}}
\newcommand{\cU}{{\mathcal U}}
\newcommand{\UU}{\mathcal{U}}

\newcommand{\cX}{{\mathcal X}}
\newcommand{\dd}{\mathbf{d}}
\newcommand{\cI}{{T}}
\newcommand{\ff}{\varphi}
\newcommand{\la}{\langle}
\newcommand{\ra}{\rangle}
\newcommand{\tr}{\mathrm{tr}}
\newcommand{\Tr}{\mathrm{Tr}}
\newcommand{\EE}{\mathrm{E}}
\newcommand{\ct}{{\bf t}}
\newcommand{\cB}{\mathcal{B}}
\newcommand{\cA}{\mathcal{A}}
\newcommand{\kt}{\kk^\ct}
\newcommand{\odo}{\otimes\cdots\otimes}
\newcommand{\cyc}{\mathbf{c}\,}
\newcommand{\SNC}{S_{NC}}
\newcommand{\SNCc}{\overline{\SNC}}
\newcommand{\SNCe}{S^{(\epsilon)}_{NC}}
\newcommand{\NC}{NC}
\newcommand{\cyclic}{{\text{\rm cyc}}}
\newcommand{\lin}{{\text{\rm lin}}}
\newcommand{\Wg}{\mathrm{Wg}}
\newcommand{\fluct}{\zeta}
\newcommand{\Nb}{{(N)}}
\newcommand{\id}{\text{id}}
\newcommand{\cS}{\mathcal{S}}
\newcommand{\ii}{\text{\bf i}}
\newcommand{\jj}{\text{\bf j}}

\hyphenation{Voi-cu-les-cu}


\newcommand{\cc}{\kappa}
\newcommand{\cG}{\mathcal{G}}
\newcommand{\cH}{\mathcal{H}}
\newcommand{\cR}{\mathcal{R}}
\newcommand{\cC}{\mathcal{C}}
\newcommand{\cov}{\mathrm{cov}}
\newcommand{\kk}{\kappa}
\newcommand{\moeb}{\mathrm{\text{M\"ob}}}
\newcommand{\cWg}{C}
\newcommand{\tovert}{\to}
\newcommand{\IZ}{\mathrm{IZ}}
\newcommand{\bF}{\text{\bf F}}
\newcommand{\Distrhigher}{D_{higher}}
\newcommand{\ab}{\allowbreak}
\newcommand{\ol}{\overline}
\newcommand{\ie}{\textit{i.e.}\,}
\newcommand{\ee}{\varepsilon}
\newcommand{\HHc}{\HH^{\circ}}

\newcommand{\knk}{{(n)}}

\newcommand{\ky}{\kappa}

\newcommand{\GUE}{\text{GUE}}

\title{Mixtures of classical and free independence}
\author{Roland Speicher}
\author{Janusz Wysoczanski}

\thanks{R.S. is supported by the ERC-Advanced Grant ``Non-commutative Distributions in Free Probability''.}

\begin{abstract}
We revive the concept of $\Lambda$-freeness of M\l otkowski
\cite{Mlo}, which describes a mixture of classical and free independence between algebras of random variables. In particular, we give a
description of this in terms of cumulants; this will be instrumental in 
the subsequent paper \cite{SW} where the quantum symmetries underlying these mixtures of classical and free independences will be considered.

\end{abstract}

\maketitle

\section{Introduction}
In the context of non-commutative probability spaces
there are only very few possibilities for universal notions of independence. If we require that this notion is commutative (i.e., $x$ independent from $y$ is the same as $y$ independent from $x$) and that
constants are independent from everything then there are only two such concepts, namely the classical independence and the free independence. On
the level of algebras, equipped with a state, this means that there are only two universal kind of product constructions, namely the tensor product and the
reduced free product. We refer the reader to \cite{Sp-universal,Mur,BGSch} for more details on this.

So if we have a collection of variables which are independent (in this univeral sense) then there are only two possibilities; they are either all classically independent or they are all freely 
independent. On the other hand, we can gain some more flexibility if we
do not ask for the same kind of independence between all of them. This 
raises the question about mixtures of the two forms of independences.
Of course, one can create quite easily such situations by starting with two sets
of variables $X$ and $Y$ which are free; then split each of them into 
two subsets $X=X_1\cup X_2$ and $Y=Y_1\cup Y_2$, such that $X_1$ and
$X_2$ are classically independent and $Y_1$ and $Y_2$ are freely independent. One can continue in this fashion and get so a collection of variables where some pairs of them are free and other pairs are classically independent. However, this is restricted to situations where we can group our
variables in sets with specific kind of independence among them. We are interested in a generalization of this, by trying to prescribe arbitrarily free or classical independence for any pair. An example for this would be to ask for five
variables $x_1,x_2,x_3,x_4,x_5$ such that
\begin{itemize}
\item
$x_1$ and $x_2$ are free, 
\item
$x_2$ and $x_3$ are free, 
\item
$x_3$ and $x_4$ are
free, 
\item
$x_4$ and $x_5$ are free, 
\item
$x_5$ and $x_1$ are free,
\item
but all other pairs are independent. 
\end{itemize}
(Here and in the following we will always mean ``classically independent'' when we say ``independent''.)

Such a situation cannot be generated
by the above dividing into groups, and it is not clear apriori whether such
a requirement can be satisfied in any meanigful way. In \cite{Mlo}
M\l otkowski showed that this can, indeed, be achieved for any prescription of the mixture of free and classical independence. For this he introduced the general notion of $\Lambda$-freeness. It seems that his work did not get the attention it deserves and we hope that our work will stimulate new interest in this concept. To his original results we will add here a description of the combinatorial structure of $\Lambda$-freeness, featuring in particular a formula for mixed moments in terms of free cumulants. This will be taken up
in the subsequent paper \cite{SW} and will lead to new forms of quantum groups, with partial commutation relations.

On the level of groups or semi-groups the prescription of commutation relations for some fixed pairs of generators is of course not new;
in the group case this goes, among others, under the names of ``right angled Artin groups'' (see \cite{Cha}), ``free partially commutative groups'' or ``trace groups'', in the case of semi-groups one talks about ``Cartier-Foata monoids'' (see \cite{foata1969problemes}) or ``trace monoids''.
Actually, there is also the notion of a corresponding mixed product of groups, which is usually called the ``graph product of groups'' and was introduced by Green in \cite{thesis}.
In a sense $\Lambda$-freeness reveals the notion of ``independence'' for the group algebras of such graph products of groups with respect to their canonical trace. We will make this connection precise in Proposition
\ref{prop:on-groups}.

Our interest in $\Lambda$-freeness
arouse out of discussions on similar constructions of the second author, on mixtures between monotone and boolean \cite{J-bm} and boolean and free independences \cite{J-bf}. 
Much motivation is also taken from recent work on bi-freeness \cite{Voi-bifree,MaNi-bifree,CNS-bifree}. Bifreeness does not fit in the frame presented here, but there are some similarities, in particular, concerning the underlying combinatorics.

\section{The setting}
The notion of $\Lambda$-freeness is defined in terms of a matrix
which specifies the choice which pairs should be free and which should be independent. 
M\l otkowski denoted this matrix by $\Lambda$; we prefer here to call it $\ee$, and hence we will also speak of $\ee$-freeness or, alternatingly, $\ee$-independence.

So let $I$ be an index set (finite or infinite). 
For any given collection of algebras $\cA_i$,
for all $i\in I$, we want to embed the $\cA_i$ in a bigger algebra $\cA$,
such that for each pair of algebras we have that they are either free or independent. In order to specify this choice we will use a symmetric matrix
$\ee=(\ee_{ij})_{i,j\in I}$ with non-diagonal entries either 0 or 1. This $\ee$ should specify our 
mixture according to: 
\begin{itemize}
\item
$\cA_i$ and $\cA_j$ are free if $\ee_{ij}=0$, and
\item
$\cA_i$ and $\cA_j$ are independent if $\ee_{ij}=1$ (which includes in
particular, that $\cA_i$ and $\cA_j$ commute
\end{itemize}
It will be convenient to set $\ee_{ii}=0$ for all $i\in I$.

In the following such a matrix $\ee$ will be fixed.
Of course, we can identify such a matrix with the adjacency matrix of a simple (i.e., no loops, no multiple edges) graph; then the edges of the graph give us the independence relations between the involved algebras, which correspond to the vertices of the graph.

For the basic notions and facts about non-commutative probability spaces, non-crossing partitions or free cumulants we refer to \cite{NS}.

\section{The definition of $\ee$-independence}

\begin{notation}
Let us use the following notation. Given some subalgebras 
$\cA_i$ ($i\in I$) and an index-tuple $\ii=(i(1),\dots,i(n))\in I^n$ we write
$(a_1,\dots,a_n)\in \cA_\ii$ for: $a_k\in \cA_{i(k)}$ for $k=1,\dots,n$.
\end{notation}

\begin{definition}
1) By $I_n^\ee$ we denote those $n$-tuples of indices from $I$ for which neigbours are different modulo our $\ee$-relations; more precisely, $\ii=(i(1),\dots,i(n))\in I_n^\ee$ if and only if: if we have $i(k)=i(l)$ for $1\leq k<l\leq n$ then there is a $p$ with
$k<p<l$ such that $i(p)\not=i(k)$ and $\ee_{i(k)i(p)}=0$.

2) Let $(\cA,\ff)$ be a non-commutative probability space. We say that unital subalgebras $\cA_i$ ($i\in I$) are \emph{$\ee$-independent}, if we have the following.
\begin{itemize}
\item
$\cA_i$ and $\cA_j$ commute for all $(i,j)$ for which $\ee_{ij}=1$ and
\item
whenever $n\in\NN$ and $(a_1,\dots,a_n)\in \cA_{\ii}$ such that
$\ff(a_k)=0$ for all $k=1,\dots,n$ and such that $\ii\in I_n^\ee$,
then we have $\ff(a_1\cdots a_n)=0$. 
\end{itemize}
\end{definition}

Note that we can use the usual centering trick to reduce any mixed moment to mixed moments of the above form; hence if we know $\ff$ restricted to each of the $\cA_i$ and we know that the $\cA_i$ are $\ee$-independent, then $\ff$ is uniquely determined on the algebra generated by all the $\cA_i$. Namely, consider an arbitrary mixed moment of the form
$\ff(a_1\cdots a_n)$ with $(a_1,\dots,a_n)\in\cA_\ii$. We can also assume that
$\ii\in I^n_\ee$ (otherwise, by using the commutation relations among the algebras, we bring elements from the same algebra together and replace them by their product). Then we write each $a_k$ as
$a_k=\ff(a_k)1+ a_k^\circ$. We plug this in for $a_1\cdots a_n$ and multiply out. We get one term of length $n$, namely $a_1^\circ a_2^\circ\cdots
a_n^\circ$ plus many other terms with fewer factors. By induction we can assume that we already know how to calculate $\ff$ applied to those smaller terms, and for the longest term we have $\ff(a_1^\circ a_2^\circ\cdots
a_n^\circ)=0$, by our definition of $\ee$-independence.

It is also clear that
if $\ee_{ij}=1$ for all $i\not=j$, then $\ee$-independence is the same as classical independence; and if $\ee_{ij}=0$ for all $i,j$, then $\ee$-independence
is the same as free independence.

What might be not so clear from this definition is whether, given non-commutative probability spaces $(\cA_i,\ff_i)$ for all $i\in I$, one can embed them in a bigger non-commutative probability space $(\cA,\ff)$ such that $\ff$ restricted to $\cA_i$ yields $\ff_i$ and such that the $(\cA_i)_{i\in I}$ are
$\ee$-independent in $(\cA,\ff)$. That this is indeed the case, for any choice of $\ee$, as well as the fact that positivity and traciality of the the involved linear functionals is preserved under such a construction,
was one of the main results of \cite{Mlo}.

\section{$\ee$-independence and the $\ee$-products of groups}

As we already mentioned in the Introduction, on the level of groups, the notion of groups with partial commutation relations is a well-known one. Actually, there is also the notion of an $\ee$-product of groups, which is usually called the graph product of groups (corresponding to the graph with adjacency matrix $\ee$) and was introduced by Green in \cite{thesis}, see also
\cite{HM}.

\begin{definition}
Let $G_i$ ($i\in I$) be groups. Then the $\ee$-product (or the \emph{graph product}) $\star_\ee G_i$ is
the quotient of the free product group $\star_{i\in I}G_i$ by the relations that $G_i$ and $G_j$ commute whenever $\ee_{ij}=1$.
\end{definition}

As expected, the notion of $\ee$-independence is adapted to this setting
of an $\ee$-product of groups.

\begin{proposition}\label{prop:on-groups}
1) Let $G=\star_\ee G_i$ be the $\ee$-product of subgroups $G_i$. Denote by
$\tau:\CC G\to\CC$ the canonical state on the group algebra $\CC G$, which 
gives the coefficient of the neutral element in a linear combination of group elements. Then, the group algebras of the subgroups, $\CC G_i$ ($i\in I$),
are $\ee$-independent in the non-commutative probability space $(\CC G,\tau)$.

2) In particular, in the group algebra of a right angled Artin group $G=\la s_i (i\in I) \mid s_is_j=s_js_i \text{ for all $i\not=j$ with $\ee_{ij}=1$}\ra$ the generators $s_i$ ($i\in I$) are $\ee$-independent.
\end{proposition}

\begin{proof}
1) The $\ee$-commutation relations are clear. So it remains 
to show that a product $g_1\cdots g_n$ of group elements with $g_j\in G_{i(j)}$ and $(i(1),\dots,i(n))\in I_n^\ee$, cannot be the neutral element if none of the $g_j$ is the neutral element. But this follows from the description of graph groups in \cite{thesis}. In the notation of Definition 3.5 of \cite{thesis},  $(g_1,\dots,g_n)$ is a reduced sequence, and then the above
statement is contained in Theorem 3.9 of \cite{thesis}. See also \cite{Cha,HM}.

2) This follows from the previous part, because our right angled Artin group is the $\ee$-product of $\vert I\vert$-many copies of $\ZZ$.
\end{proof}

\section{Description of $\ee$-independence via free cumulants}

We come now to the main result of this note.
Namely, we want to see that we can also describe our notion of $\ee$-independence by some cumulant machinery. Note, however, that we do not 
introduce some kind of new cumulants, but the moment-cumulant formula will always involve the usual free cumulants. What makes the difference is the set of partitions over which we sum. 

\begin{definition}
Let us define, for each $\ii=(i(1),\dots,i(n))$,
$NC^\ee[\ii]$ as those partitions $\pi\in \cP(n)$ for which we have $\pi\leq
\ker\ii$ (i.e., $\pi$ connects only $k$ and $l$ for which we have $i(k)=i(l)$) and
which can be reduced to the empty partition by
iteration of the following two operations: 
\begin{itemize}
\item
remove ``interval''-blocks, which consist just of neighbouring elements; i.e.,
if $\pi=\tilde \pi\cup\{(r,r+1,r+2,\dots,r+p)\}$, then $\pi\in NC^\ee[\ii]$ if
and only if $\tilde \pi\in NC^\ee[i(1),\dots,i(r-1),i(r+p+1),\dots,i(n)]$
\item
exchange the points $k$ and $k+1$ if we have
$\ee_{i(k)i(k+1)}=1$; i.e., if we denote by $\pi_{l\leftrightarrow k}$ the partition which we get from $\pi$ by swaping the points $k$ and $l$, then
$$\pi\in NC^\ee[i(1),\dots,i(k),i(k+1),\dots,i(n)]$$
if and only if
$$\pi_{k\leftrightarrow k+1}\in NC^\ee[i(1),\dots,i(k+1),i(k),\dots,i(n)].$$
Recall that on the diagonal we have set $\ee$ to 0, i.e., we have $\ee_{ii}=0$ for all $i\in I$.
\end{itemize}
\end{definition}

Another way of saying this is
$$NC^\ee[\ii]=\{ \pi\in \cP(n)\mid \pi \leq \ker \ii \quad\text{and $\pi$ is
$(\ee,\ii)$-non-crossing}\},$$
where $(\ee,\ii)$-non-crossing for a $\pi$ with $\pi\leq \ker\ii$ means that if there are $1\leq p_1<q_1<p_2<q_2\leq n$ such that
$p_1\sim_\pi p_2$, $q_1\sim_\pi q_2$, $p_1\not\sim_\pi q_1$, then $\ee_{i(p_1)i(q_1)}=1$.

Note that in the case where all $i$-indices are the same, $i(1)=i(2)=\dots=
i(n)=i$, the second operation comes never into effect and hence, for any choice of $\ee$, we have
$$NC^\ee[(i,i,\dots,i)]=NC(n).$$

Let us also check the two extremes in $\ee$. First, assume that all $\ee_{ij}$ are zero. Then $(\ee,\ii)$-non-crossing is the same as non-crossing and hence we have:
\begin{equation}\label{eq:NC-zero}
NC^\ee[\ii]=\{\pi\in NC(n)\mid \pi\leq\ker \ii\}
 \qquad\text{if $\ee_{ij}=0$ for all $i,j$}.
\end{equation}

On the other hand,  when
$\ee_{ij}=1$ for all $i\not=j$, then all blocks of $\ker\ii$ can be commuted and
$NC^\ee[\ii]$ factorizes into a product of non-crossing lattices, one for each block of $\ker \ii$,
\begin{equation}\label{eq:NC-fact}
NC^\ee[\ii]=\prod_{V\in\ker \ii} NC(V) \qquad\text{if $\ee_{ij}=1$ for all $i\not= j$}.
\end{equation}

\begin{theorem}\label{thm:mom-cum}
Let $\cA_i$ ($i\in I$) be $\ee$-independent in $(\cA,\ff)$. Consider $\ii\in I^n$ and  $(a_1,\dots,a_n)\in \cA_\ii$.
Then we have
\begin{equation}\label{eq:m-c}
\ff(a_1\cdots a_n)=\sum_{\pi\in NC^\ee[\ii]}\kk_\pi(a_1,\dots,a_n),
\end{equation}
where  
$\kk_\pi(a_1,\dots,a_n)$ is the product of
the {\it free} cumulants for each block,
$$\kk_\pi(a_1,\dots,a_n)=\prod_{V\in\pi} \kk_V((a_k)\vert V),$$
where for $V=(r_1<\dots < r_p)\in\pi$ we set
$$ \kk_V((a_k)\vert V)=\kk_p(a_{r_1},\dots,a_{r_p}).$$
\end{theorem}

Let us first check that this formula is the correct one in the two extreme cases where all pairs have the same kind of independence. Assume first that all $\ee_{ij}=0$. Then $NC^\ee[\ii]$ is always $[0,\ker\ii]\subset NC(n)$ and the formula is just the moment-cumulant formula in the free case, combined with the fact that
our restriction to the summation $\pi\leq\ker\ii$ amounts to the vanishing of mixed free cumulants. This gives then the rule for the calculation of free random variables.

Consider now the other extreme that $\ee_{ij}=1$ for all $i\not= j$. Then $NC^\ee[\ii]$ factorizes as in \eqref{eq:NC-fact}, and \eqref{eq:m-c} is then
$$\ff(a_1\cdots a_n)=
\sum_{\pi=(\pi_V)_{V\in\ker\ii}}\prod_{V\in\ker\ii} \kk_{\pi_V}((a_i)\vert V)=
\prod_{V\in\ker\ii} \ff((a_k)\vert V),$$
i.e., $\ff(a_1\cdots a_n)$ factorizes into the product of the expectations of the product of the variables belonging to the same algebra. This is the rule for the calculation of independent random variables.

Note that for the previous calculation we actually only needed that all algebras for which we have a crossing in $\ker \ii$
commute. Hence the same arguments prove also the
following (which was also shown in \cite{Mlo}).

\begin{corollary}
Let $\cA_i$ ($i\in I$) be $\ee$-independent in $(\cA,\ff)$. Consider a mixed moment $\ff(a_1\cdots a_n)$ for $(a_1,\dots,a_n)\in \cA_\ii$ with $\ii\in I^n$. If $\ker \ii\in NC^\ee[\ii]$ then the mixed moment factorizes into the product
$$\ff(a_1\cdots a_n)=
\prod_{V\in\ker\ii} \ff((a_k)\vert V).$$
\end{corollary}

\begin{proof}
For each $i\in I$, let $(\cB_i,\psi_i)$ be a copy of $(\cA_i,\ff_i)$ and define $\cB$ as the free product of the $\cB_i$ with amalgamation over $\CC 1$; i.e., we identify the units of the $\cB_i$, but have no further relation among different $\cB_i$'s. Hence in $\cB$ we have that $\cB_i\cap \cB_j=\CC 1$ for all $i\not= j$.

We define  
$$\psi^\knk:\bigcup_{\ii\in I^n}\cB_\ii\to\CC$$ by
\begin{equation}\label{eq:m-c-b}
\psi^\knk(b_1,\dots ,b_n)=\sum_{\pi\in NC^\ee[\ii]}\kk_\pi(b_1,\dots,b_n)\qquad
\text{for $(b_1,\dots,b_n)\in\cB_\ii$}.
\end{equation}
The $\kk_\pi(b_1,\dots,b_n)$ are here as before the product of the free cumulants corresponding to the blocks of $\pi$, and for each block we use the free cumulants given by the corresponding $\psi_i$. The only ambiguity in the definition \eqref{eq:m-c-b} might occur when some of the $b_k$ belong to several $\cB_i$. However, this can only happen for multiples of 1. Let us check the case where $b_1=1$, so that we have $(b_1,b_2,\dots,b_n)\in\cB_\ii$
for $\ii=(i,i(2),\dots,i(n)$ for arbitrary $i$. We have to see that the formula in \eqref{eq:m-c-b} is independent of $i$. But this follows from the fact that $\kk_\pi(1,b_2,\dots,b_n)$ is zero unless the first element is a singleton, hence $\pi$ must be of the form $\pi=(1)\cup\sigma$, where $\sigma\in\cP(2,\dots,n)$. But in the constraint $(1)\cup\sigma\in NC^\ee[(i,i(2),\dots,i(n)]$ the value of $i$ does not play a role, since the block $(1)$ cannot have any crossings. Thus our $\psi^\knk$'s are well-defined. We will use them to define a functional $\psi$ on $\cB$ by putting
$$\psi(b_1\cdots b_n):=\psi^\knk(b_1,\dots ,b_n)\qquad
\text{for $(b_1,\dots,b_n)\in\cB_\ii$}$$
and extend this linearly.
Again we have to make sure that this is well-defined; we have to check that in the situation where two neighbouring $b_k$'s, say $b_1$ and $b_2$ come from the same algebra,
both possible definitions give the same, i.e.,
for $(b_1,\tilde b_1,b_2\dots,b_{n})\in\cB_{[(i(1),i(1),i(2),\dots,i(n))]}$ we 
must have
\begin{equation*}
\psi^{(n+1)}(b_1,\tilde b_1,b_2\dots,b_n)=\psi^\knk(b_1\tilde b_1,b_2,\dots,b_n).
\end{equation*}
The left hand side is given by
\begin{equation}\label{eq:psi-LHS}
\psi^{(n+1)}(b_1,\tilde b_1,b_2\dots,b_n)=\sum_{\pi\in NC^\ee[(i(1),i(1),i(2),\dots,i(n))]} \kk_\pi(b_1,\tilde b_1,b_2,\dots,b_n),
\end{equation}
whereas the right hand side is given by
\begin{equation}\label{eq:psi-RHS}
\psi^\knk(b_1\tilde b_1,b_2,\dots,b_n)=\sum_{\sigma\in NC^\ee[(i(1),i(2),\dots,i(n))]}\kk_\sigma(b_1\tilde b_1,b_2,\dots,b_n).
\end{equation}
The cumulant corresponding to the first block $V=(1<r(1)<\dots<r(p))$ of $\sigma$ is now, by the formula for free cumulants with products as arguments (see Theorem 11.12 in \cite{NS}), the same as
\begin{multline*}
\kk_{p+1}(b_1\tilde b_1,b_{r(1)},\dots,b_{r(p)})=\kk_{p+2}(b_1,\tilde b_1,b_{r(1)},\dots,b_{r(p)})\\+\sum_{q=0}^r\kk_{p-q+1}(b_1,b_{r(q+1)},b_{r(q+2)},\dots,b_{r(p)})\cdot
\kk_{q+1}(\tilde b_1,b_{r(1)},\dots,b_{r(q)})
\end{multline*}
These terms correspond exactly to the contributions of those $\pi$ in 
\eqref{eq:psi-LHS}, which collapse to $\sigma$ under the identification of the first two elements. This shows that \eqref{eq:psi-LHS} and \eqref{eq:psi-RHS} agree and our $\psi$ is well-defined on $\cB$.

We claim now that this $\psi$ satisfies the defining property of $\ee$-independence. Assume we have
$(b_1,\dots,b_n)\in\cB_\ii$ with $\ii\in I^\ee_n$ and such that $\psi(b_k)=0$ for all $k=1,\dots,n$. But then the definition of $I^\ee_n$ and $NC^\ee[\ii]$ imply that every $\pi\in NC^\ee[\ii]$ must have at least one singleton, which means that the corresponding contribution $\kk_\pi$ in \eqref{eq:m-c-b} is zero; hence 
$$\psi(b_1\cdots b_n)=\psi^{(n)}(b_1,\dots,b_n)=0.$$
Since $\ee$-independence and the distribution on the individual algebras determines the distribution on the generated algebra, $\psi$ must agree, via the canonical identification $\cB_i\to\cA_i$,
with $\ff$ on the algebra generated by the $\cA_i$; hence the formula \eqref{eq:m-c-b} is also valid for $\ff$.
\end{proof}

\begin{remark}\label{rem:Guillaume}
One might wonder about the apparent unsymmetry of Theorem \ref{thm:mom-cum} with respect to free and classical independence, as
only free cumulants show up. However, as was pointed out to us by Guillaume Cebron this is due to our choice that on the diagonal $\ee_{ii}$ is always zero; which results in the fact that each variable is described in terms of its free cumulants. We could also change this convention and put all $\ee_{ii}=1$; then each variable goes with classical cumulants and we get a version of Theorem \ref{thm:mom-cum} where the classical cumulants instead of the free cumulants show up. Of course, the set $NC^\ee$ is then different, in particular, with this definition we would have $NC^\ee[(i,i,\dots,i)]=\cP(n)$.
Also mixtures between free and classical cumulants are possible, by choosing some $\ee_{ii}=0$ and other $\ee_{jj}=1$.
\end{remark}

\section*{Acknowledgements}
We thank Franz Lehner and Guillaume Cebron for discussions; in particular, the former for pointing out the relevance of Coxeter and Artin groups in this context and the latter for Remark \ref{rem:Guillaume}.

\bibliographystyle{plain}
\bibliography{Speicher}

\end{document}